\documentclass[11pt, a4paper]{article}
\usepackage{amssymb}

\usepackage{amsfonts}
\usepackage{amsmath}
\usepackage{fullpage}
\usepackage[all]{xy}

\newcommand{\presym}[1]{S_{\eta^{*}}^{(I,J)}(#1)}

\newtheorem{theorem}{Theorem}

\newtheorem{conjecture}{Conjecture}
\newtheorem{corollary}{Corollary}

\newtheorem{lemma}[theorem]{Lemma}

\newtheorem{proposition}[theorem]{Proposition}

\newenvironment{proof}[1][Proof]{\begin{trivlist}
\item[\hskip \labelsep {\bfseries #1}]}{\end{trivlist}}
\begin{document}
\title{Some Properties of Macdonald Polynomials with \\ Prescribed Symmetry.}
\author{Wendy Baratta}
\maketitle

\begin{abstract}
The Macdonald polynomials with prescribed symmetry are obtained from the nonsymmetric Macdonald polynomials via the operations of $t$-symmetrisation, $t$-antisymmetrisation and normalisation. Motivated by corresponding results in Jack polynomial theory we proceed to derive an expansion formula and a related normalisation. Eigenoperator methods are used to relate the symmetric and antisymmetric Macdonald polynomials, and we discuss how these methods can be extended to special classes of the prescribed symmetry polynomials in terms of their symmetric counterpart. We compute the explicit form of the normalisation with respect to the constant term inner product. Surpassing our original motivation, this is used to provide a derivation of a special case of a conjectured $q$-constant term identity. \end{abstract}
\subsubsection*{}
\small{\textbf{2000 Mathematics Subject Classification:} Primary 81Q08, Secondary 81V08. \\ \textbf{Keywords:} Macdonald polynomials; prescribed symmetry; constant term identities.}

\normalsize
\section{Introduction\label{introduction}}
\subsection{Background and overview}
Nonsymmetric Macdonald polynomials were first introduced in 1994 \cite{affine, cherednik}, six years after Macdonald's paper \cite{macdonaldagain} introducing what are now referred to as symmetric Macdonald
polynomials $P_{\kappa }\left( z;q,t\right) .$ The nonsymmetric Macdonald
polynomials $E_{\eta }\left( z;q,t\right)$ can be regarded as building blocks of their symmetric
counterparts, as $t$-symmetrisation of $E_{\eta }$ gives $P_{\eta ^{+}}$. Generalising this action by applying a combination of $t$-symmetrising and
$t$-antisymmetrising operators to $E_{\eta }$ generates the Macdonald polynomials with
prescribed symmetry. \bigskip \\
A polynomial is $t$-symmetric with respect to $z_i$ if $T_if(z)=tf(z)$ and $t$-antisymmetric with respect to $z_i$ if $T_if(z)=-f(z)$. The $t$-symmetrisation and $t$-antisymmetrisation operators are defined, respectively, by
\begin{equation}
U^{+}:=\sum_{\sigma \in S_{n}}T_{\sigma },\;\;\;\;\text{and}\;\;\;\;
U^{-}:=\sum_{\sigma \in S_{n}}\left(-\frac{1}{t}\right)^{l\left( \sigma \right)
}T_{\sigma }. \label{1a}
\end{equation}%
Here $S_{n}$ denotes the symmetric group on $n$ symbols. Also, with $s_i$ denoting the transposition operator with the action on functions 
\begin{equation*}
s_{i}f\left( z_{1},...,z_{n}\right) =f\left(
z_{1},...,z_{i+1},z_{i},...,z_{n}\right), \hspace{1cm}
(i=1,...,n-1)
\end{equation*}
and $\sigma :=s_{i_{l(\sigma)}}...s_{i_{1}},$ the operator $T_{\sigma }$ is
specified by%
\begin{equation}
T_{\sigma }:=T_{i_{l(\sigma)}}...T_{i_{1}},  \label{tstring}
\end{equation}%
where 
\begin{equation*}
T_{i}:=t+\frac{tz_{i}-z_{i+1}}{z_{i}-z_{i+1}}(s_{i}-1),\hspace{1cm}
(i=1,...,n-1).
\end{equation*}%
Note that when $t=1$ the operators $U^{+}$ and $U^{-}$ reduce to the standard
symmetrising and antisymmetrising operators. \bigskip \\
The study of Macdonald polynomials with prescribed symmetry began in 1999 \cite{exchangeterms, marshallthesis} and was initially motivated by the analogous results in Jack polynomial theory \cite{genweight, peter4}. Nonsymmetric Jack polynomials $E_{\eta}(z;\alpha)$ are the limit $q=t^{\alpha}$, $t\rightarrow1$ of Macdonald polynomials. Jack polynomials are eigenfunctions of the operator stemming from the type $A$ Calogero-Sutherland quantum many body system. Cases where the system is multicomponent, containing both bosons and fermions, requires eigenfunctions that are symmetric or antisymmetric, respectively, with respect to certain sets of variables. This requirement lead naturally to the introduction of Jack polynomials with prescribed symmetry \cite{peter4,kato}.
\bigskip \\
Using the more extensively developed Jack theory as motivation we continue the study of prescribed symmetry Macdonald polynomials. Our first result is an expansion formula for the prescribed symmetry Macdonald polynomials in terms of the nonsymmetric Macdonald polynomials (Proposition \ref{proposition 2.1}). Following this we determine the normalisation required to obtain the prescribed symmetry Macdonald polynomial from the symmetrisation of the nonsymmetric polynomial (Proposition \ref{proposition 2.2}). In Section \ref{special forms} eigenoperator methods are used to relate the symmetric and antisymmetric Macdonald polynomials, thus providing an alternate proof for a result of Marshall \cite{marshallmacdonald}. Our final investigation is of the inner product of prescribed symmetry Macdonald polynomials, giving general explicit formulas (Theorem \ref{innerstuff}) and then considering special cases where the antisymmetric components are of specific forms. Although originally motivated by the analogous Jack theory \cite{exchangeterms} these results have applications in $q$-constant terms identities, which are discussed in the second half of Section \ref{prescribed symmetry
polynomials and the inner product}.
\bigskip \\
Before explicitly defining the Macdonald polynomials with prescribed symmetry, we give the required
background on the nonsymmetric Macdonald polynomials. 
\subsection{Nonsymmetric Macdonald polynomials}

The nonsymmetric Macdonald polynomials $E_{\eta }:=E_{\eta }\left( z;q,t\right) $ are
polynomials of $n$ variables $z=\left( z_{1},...,z_{n}\right) $ having
coefficients in the field $\mathbb{Q}\left( q,t\right) $ of rational
functions of the indeterminants $q$ and $t$. The compositions $\eta :=\left(
\eta _{1},...,\eta _{n}\right) $ of non-negative integers components $\eta _{i}$
label these polynomials. The nonsymmetric Macdonald polynomials can be
defined, up to normalisation, as the unique simultaneous polynomials eigenfunctions of
the commuting operators 
\begin{equation}
Y_{i}=t^{-n+1}T_{i}...T_{n-1}\omega T_{1}^{-1}...T_{i-1}^{-1},\hspace{1cm}
(i=1,...,n)  \label{Yi}
\end{equation}%
satisfying the eigenvalue equations 
\begin{equation}
Y_{i}E_{\eta }\left( z;q,t\right) =\overline{\eta }_{i}E_{\eta }\left(
z;q,t\right) .  \label{eigen}
\end{equation}%
In (\ref{Yi}) $\omega $ is given by $\omega :=s_{n-1}...s_{1}\tau _{1}$, where the operator $\tau _{i}$ has the action on functions 
\begin{equation*}
\left( \tau _{i}f\right) \left( z_{1},...,z_{n}\right) :=f\left(
z_{1},...,qz_{i},...,z_{n}\right)
\end{equation*}%
and so corresponds to a $q$-shift of the variable $z_{i}.$ The eigenvalue $%
\overline{\eta }_{i}$ in $\left( \ref{eigen}\right) $ is given by 
\begin{equation}
\overline{\eta }_{i}:=q^{\eta _{i}}t^{-l_{\eta }^{\prime }\left( i\right) },
\label{eigenvalue}
\end{equation}%
where%
\begin{equation}
l_{\eta }^{\prime }(i):=\#\left\{ j<i;\eta _{j}\geq \eta _{i}\right\}
+\#\left\{ j>i;\eta _{j}>\eta _{i}\right\} . \label{legcolength}
\end{equation}%
The nonsymmetric Macdonald polynomials are of the triangular form 
\begin{equation}
E_{\eta }\left( z;q,t\right) :=z^{\eta }+\sum\limits_{\nu \prec \eta
}b_{\eta \nu }z^{\nu },  \label{forma}
\end{equation}%
for coefficients $b_{\eta \nu }\in\mathbb{Q}\left( q,t\right) $. The coefficient of $z^{\eta}:=z_{1}^{\eta _{1}}...z_{n}^{\eta _{n}}$ is chosen to be unity as a normalisation. The ordering $\prec $ is a partial
ordering on compositions having the same modulus, where $\left\vert \eta
\right\vert :=\Sigma _{i=1}^{n}\eta _{i}$ denotes the modulus of $\eta $.
The partial ordering is defined by%
\begin{equation*}
\mu \prec \eta \text{ iff }\mu ^{+}<\eta ^{+}\text{ or in the case }\mu
^{+}=\eta ^{+},\text{ }\mu <\eta
\end{equation*}%
where $\eta ^{+}$ is the unique partition obtained by permuting the
components of $\eta$, and $\mu <\eta $ iff $\mu \neq \eta $ and $\Sigma
_{i=1}^{p}\left( \eta _{i}-\mu _{i}\right) \geq 0$ for all $1\leq p\leq n.$
The action of $T_{i}$ on $E_{\eta }$, for $1\leq i \leq n-1$, is given explicitly by \cite{peter3} 
\begin{equation}
T_{i}E_{\eta }\left( z\right) =\left\{ 
\begin{tabular}{ll}
$\frac{t-1}{1-\delta _{i,\eta }^{-1} } E_{\eta
}+tE_{s_{i}\eta }$, & $\eta _{i}<\eta _{i+1}$ \\ 
$tE_{\eta }$, & $\eta _{i}=\eta _{i+1}$ \\ 
$\frac{t-1}{1-\delta _{i,\eta }^{-1} } E_{\eta }+\frac{( 1-t\delta _{i,\eta })( 1-t^{-1}\delta _{i,\eta })}{\left( 1-\delta _{i,\eta }\right)^{2}}E_{s_{i}\eta}$

& $\eta _{i}>\eta _{i+1}$,%
\end{tabular}%
\right.  \label{TiEn}
\end{equation}%
where $\delta _{i,\eta }=\overline{\eta _{i}}/\overline{%
\eta} _{i+1}.$ 
\bigskip \\
Alternatively, nonsymmetric Macdonald polynomials can be characterised as multivariate polynomials of the structure (\ref{forma}) orthogonal with respect to the inner product $\left\langle f,g\right\rangle _{q,t}$, defined by
\begin{equation}
\left\langle f,g\right\rangle _{q,t}:=\text{CT}\left( f( z;q,t)
g( z^{-1},q^{-1},t^{-1}) W\left( z\right) \right).
\label{inner product}
\end{equation}%
In (\ref{inner product}) CT$\left( f\right) $ denotes the constant term with respect to $z$ of
any formal Laurent series $f$ and 
\begin{equation*}
W\left( z\right) :=W\left( z;q,t\right) :=\prod_{1\leq i<j\leq n}\frac{%
\left( \frac{z_{i}}{z_{j}};q\right) _{\infty }\left( q\frac{z_{j}}{z_{i}}%
;q\right) _{\infty }}{\left( t\frac{z_{i}}{z_{j}};q\right) _{\infty }\left(
qt\frac{z_{j}}{z_{i}};q\right) _{\infty }},
\end{equation*}%
with the Pochhammer symbol defined by $\left( a;q\right) _{\infty }:=\prod_{j=0}^{\infty }\left( 1-aq^{j}\right)$.
\bigskip \\
This scalar product, introduced by Cherednik, is linear and positive
definite. Macdonald showed that  \cite{affine}
\begin{equation*}
\left\langle E_{\eta }\left(q,t\right) ,E_{\nu }\left( q,t\right) \right\rangle _{q,t}=\delta_{\eta,\nu}\text{ }\mathcal{N}_{\eta}. \label{innerorth}
\end{equation*}
We will have future use for the explicit value of $\mathcal{N}_{\eta}$. For this a number of quantities dependent on $\eta$ must be introduced. For each node $s=(i,j)\in \text{diag}(\eta)$ we define the arm length, $a_{\eta}(s):=\eta_{i}-j$, arm colength, $a'_{\eta}(s):=j-1$, leg length, $l_{\eta}(s):=\#\{k<i:j\leq\eta_{k}+1\leq\eta_{i}\}+\#\{k<i:j\leq\eta_{k}\leq\eta_{i}\}$ and leg colength $l'_{\eta}(s)$, given by (\ref{legcolength}). From these we define \cite{sahi}
\begin{align*}
d_{\eta }&:=d_{\eta }\left( q,t\right)=\prod_{s\in \text{diag}(\eta )}\left(
1-q^{a_{\eta }\left( s\right) +1}t^{l_{\eta }\left( s\right) +1}\right) 
\text{, \ \ } \\
d_{\eta }^{\prime }&:=d_{\eta }^{\prime }\left( q,t\right) =\prod_{s\in \text{diag}%
(\eta )}\left( 1-q^{a_{\eta }\left( s\right) +1}t^{l_{\eta }\left( s\right)
}\right) , \\
e_{\eta }&:=e_{\eta }\left( q,t\right) =\prod_{s\in \text{diag}(\eta )}\left(
1-q^{a_{\eta }^{\prime }\left( s\right) +1}t^{n-l_{\eta }^{\prime }\left(
s\right) }\right) \text{ }, \\
e_{\eta }^{\prime }&:=e_{\eta }^{\prime }\left( q,t\right) =\prod_{s\in \text{diag}(\eta
)}\left( 1-q^{a_{\eta }^{\prime }\left( s\right) +1}t^{n-1-l_{\eta }^{\prime
}\left( s\right) }\right).
\end{align*}
In this notation the explicit formula for $\mathcal{N}_\eta$ is given by [see, e.g., \ref{cherednik}]
\begin{equation}
\mathcal{N}_\eta=\frac{d'_{\eta } e_{\eta
}}{d_{\eta }e_{\eta }^{\prime }}{\left\langle
1,1\right\rangle _{q,t}}. \label{12today}
\end{equation}
In later sections we use the specialisation $t=q^{k},\text{ } k \in \mathbb{Z} ^{+}$. In this specialisation the weight function $W(z)$ reduces to 
\begin{equation*}
W(z;q,q^{k})=\prod_{1\leq i<j\leq n}\left(\frac{z_{i}}{z_{j}};q\right)_{k}\left(q\frac{z_{j}}{z_{i}};q\right)_{k},
\label{new weight}
\end{equation*}
where
\begin{equation}
(a;q)_k:=\Pi_{j=0}^{k-1}(1-aq^j), \hspace{0.5cm} k>0 \label{9a}
\end{equation} 
which is a Laurent polynomial. Furthermore \cite{genweight}
\begin{equation}
\left\langle 1,1\right\rangle _{q,q^{k}}=\frac{[nk]_{q}!}{[k]_{q}!^{n}}, \label{13a}
\end{equation}%
where the $q$-factorial is given in terms of $q$-numbers \cite{qtheory} $[m]_{q}:=(1-q^{m})/(1-q)$ by 
\begin{equation*}
\left[ k\right] _{q}! :=[1]_{q}[2]_{q}\ldots[k]_{q}.  \label{q factorial}
\end{equation*}
Concluding the preliminary material, we now proceed to formally introduce the Macdonald polynomials with prescribed symmetry.

\section{Macdonald Polynomials with Prescribed Symmetry and the Required Operator}

\subsection{The operator O$_{I,J}$}

We begin our investigation into the Macdonald polynomials with prescribed symmetry by introducing a particular symmetrising operator $O_{I,J}$. The sets $I$ and $J$ represent the variables which the operator $O_{I,J}$ symmetrises and antisymmetrises with respect to. Explicitly \begin{equation}
T_{i}[O_{I,J}f\left( z\right) ]=tO_{I,J}f\left( z\right) \text{ for }i\in I  
\label{symmetric oij}
\end{equation}%
and%
\begin{equation}
T_{j}[O_{I,J}f\left( z\right) ]=-O_{I,J}f\left( z\right) \text{ for }j\in J.
\label{asymmetric oij}
\end{equation}%
\bigskip \\%
 For $O_{I,J}$ to be well defined $I$ and $J$ must be
disjoint subsets of $\left\{ 1,...,n-1\right\} $, such that%
\begin{equation*}
i-1,i+1\not\in J\text{ for }i\in I\text{ and }j-1,j+1\not\in I\text{ for }%
j\in J.
\end{equation*}%
In many cases we require the set $J$ to be decomposed into disjoint sets of consecutive integers, to be denoted $J_1,J_2,\ldots,J_s$. For example, with $J=\{1,2,5,6,7\}$, $J_1=\{1,2\}$ and $J_2=\{5,6,7\}$. Related to this we also require sets $\widetilde{J}_{j}:=J_{j}\cup \left\{ \max
\left( J_{j}\right) +1\right\} $ and $\widetilde{J}=\cup \widetilde{J}_{s}.$
\bigskip \\
Since we are symmetrising with respect to a subset of variables, in contrast to the construction of $U^+$ and $U^-$ given in (\ref{1a}), we do not want to
sum over all $\sigma \in S_{n}$. Instead we introduce $W_{I\cup J}:=\left\langle s_{k};\text{ }k\in I\cup
J\right\rangle ,$ a subset of $S_{n}$ such that each $\omega \in W_{I\cup J}$,
\begin{equation}
\omega =\omega _{I}\omega _{J}, \hspace{0.25cm} \text{ with }\omega _{J}\in W_{J} 
\text{ and } 
\omega _{I}\in W_{I}, \label{12'} 
\end{equation}
has the property that $\omega \left( i\right) =i$ if $i\not\in \widetilde{I}%
\cup \widetilde{J}.$%
\bigskip \\%
The operator $O_{I,J}$ is then specified by 
\begin{equation*}
O_{I,J}:=\sum_{\omega \in W_{I\cup J}}\left( -\frac{1}{t}\right) ^{l\left(
\omega _{J}\right) }T_{\omega }, \label{12a}
\end{equation*}
where $T_\omega$ is given by (\ref{tstring}).%

\subsection{The polynomial $\presym{z}$}
To motivate the introduction of the prescribed symmetry Macdonald polynomials we first consider the symmetric and antisymmetric Macdonald polynomials. These are denoted by $P_\kappa$ and $S_{\lambda+\delta}$, respectively, where $\kappa$ and $\lambda$ are partitions and $\delta:=(n-1,\ldots,1,0)$. It is well known that these polynomials can be generated from nonsymmetric Macdonald polynomials via a process of symmetrisation and antisymmetrisation. Thus to generate $P_\kappa$ one would symmetrise any $E_\eta$ for which there exists a permutation $\sigma \in S_n$ such that $\sigma \eta = \kappa$. Similarly, to generate $S_{\lambda+\delta}$ one would antisymmetrise any $E_\mu$ such that there exists a permutation $\rho \in S_n$ where $\rho \mu = \lambda + \delta$. Explicitly\begin{equation*}
U^+E_\eta=b_\eta P_\kappa \hspace{0.5cm}\text{and} \hspace{0.5cm} U^-E_\mu=b'_\mu S_{\lambda+\delta},
\end{equation*}
for some non-zero $b_\eta, b'_\mu\in\mathbb{Q}\left( q,t\right) $.
It follows quite naturally that the Macdonald polynomial with prescribed symmetry, denoted by $
S_{\eta ^{\ast }}^{\left( I,J\right) }\left( z;q,t\right)$, a polynomial $t$-symmetric with respect to the set $I$ and $t$-antisymmetric with respect to the set $J$, will be labeled by a composition $\eta^*$ such that 
\begin{equation*}
\eta _{i}^{\ast }\geq \eta _{i+1}^{\ast }\text{ for all }i\in I\text{ and }%
\eta _{j}^{\ast }>\eta _{j+1}^{\ast }\text{ for all }j\in J.
\end{equation*}
Such a polynomial can be generated by applying our prescribed symmetry operator $O_{I,J}$ to any $E_\eta$ such that there exists a $\sigma \in W_{I \cup J}$ with $\sigma \eta=\eta^*$. That is   
\begin{equation}
O_{I,J}E_{\eta }\left( z;q,t\right) =a_{\eta }^{\left( I,J\right) }S_{\eta
^{\ast }}^{\left( I,J\right) }\left( z\right),  \label{2.13}
\end{equation}%
for some non-zero $a_{\eta }^{\left( I,J\right) }$. This uniquely specifies  $%
S_{\eta ^{\ast }}^{\left( I,J\right) }\left( z\right) $ up to normalisation; for the latter we require that the coefficient of $z^{\eta^*}$ in the monomial expansion equals unity as in (\ref{forma}).
\bigskip \\
Our first task is to find the explicit formula for the proportionality $a_{\eta }^{\left( I,J\right) }$ in (\ref{2.13}). We do this by first computing the expansion formula of $S_{\eta^{*}}(z^{-1};q^{-1},t^{-1})$ in terms of $E_{\eta}(z^{-1};q^{-1},t^{-1})$, a result which is of independent interest. We begin by deriving an explicit formula for the action of $T_{i}$ on $E_{\eta}(z^{-1};q^{-1},t^{-1})$ analogous to (\ref{TiEn}). This is done using the Cauchy formula for the nonsymmetric Macdonald polynomials, \cite{mimachi}
\begin{equation*} 
\Omega( x,y;q,t) :=\sum_{\eta }\frac{d_{\eta }}{d_{\eta
}^{\prime }}E_{\eta }\left( x;q,t\right) E_{\eta }(
y;q^{-1},t^{-1}),   \label{cauchy gen function}
\end{equation*}%
the result \cite{mimachi}
\begin{equation}
T_{i}^{(x)}\Omega( x,y^{-1};q,t) =T_{i}^{(y)}\Omega(
x,y^{-1};q,t). \label{3.4.1}
\end{equation}%
and (\ref{TiEn}) itself. In (\ref{3.4.1}) the superscripts denote which variables the respective operators act
upon.

\begin{proposition}
\label{TiEnInv}For $1\leq i\leq n-1$ we have 
\begin{eqnarray}
T_{i}E_{\eta }\left( z^{-1};q^{-1},t^{-1}\right) \hspace{12cm} \notag \\ =\left\{ 
\begin{tabular}{ll}
$\frac{t-1}{1-\delta _{i,\eta }^{-1}} E_{\eta }(
z^{-1};q^{-1},t^{-1}) +\frac{d_{s_{i}\eta }d_{\eta }^{\prime }}{d_{s_{i}\eta
}^{\prime }d_{\eta }}\frac{( 1-t\delta _{i,\eta }
)( 1-t^{-1}\delta _{i,\eta }) }{(
1-\delta _{i,\eta }) ^{2}}E_{s_{i}\eta }(
z^{-1};q^{-1},t^{-1}) $ & $\eta _{i}<\eta _{i+1}$ \\ 
$tE_{\eta }( z^{-1};q^{-1},t^{-1}) $ & $\eta _{i}=\eta _{i+1}$ \\ 
$\frac{t-1}{1-\delta _{i,\eta }^{-1}} E_{\eta }(
z^{-1};q^{-1},t^{-1}) +t\frac{d_{s_{i}\eta }d_{\eta }^{\prime }}{d_{s_{i}\eta
}^{\prime }d_{\eta }}E_{s_{i}\eta }( z^{-1};q^{-1},t^{-1}) $ & $\eta _{i}>\eta
_{i+1}$%
\end{tabular}%
,\right.  \label{TiEnInv1}
\end{eqnarray}
\end{proposition}

\begin{proof}
By (\ref{3.4.1}) we have 
\begin{eqnarray*}
T_{i}^{\left( x\right) }\left( \frac{d_{\eta }}{d_{\eta }^{\prime }}%
E_{\eta }( x;q,t) E_{\eta }( y^{-1};q^{-1},t^{-1}) +%
\frac{d_{s_{i}\eta }}{d_{s_{i}\eta }^{\prime }}E_{s_{i}\eta }(
x;q,t) E_{s_{i}\eta }( y^{-1};q^{-1},t^{-1}) \right)\hspace{1cm} \\
\hspace{1cm}=T_{i}^{\left( y\right) }\left( \frac{d_{\eta }}{d_{\eta }^{\prime }}%
E_{\eta }( x;q,t) E_{\eta }( y^{-1};q^{-1},t^{-1}) +%
\frac{d_{s_{i}\eta }}{d_{s_{i}\eta }^{\prime }}E_{s_{i}\eta }(
x;q,t) E_{s_{i}\eta }( y^{-1};q^{-1},t^{-1}) \right) .
\end{eqnarray*}%
Using $\left( \ref{TiEn}\right) $ and equating coefficients of like terms
gives $\left( \ref{TiEnInv1}\right) .$
\hfill $\square$ \end{proof}
The coefficients in the expansion of $S_{\eta ^{\ast }}^{\left( I,J\right)
}( z^{-1};q^{-1},t^{-1}) $ in terms of $\left\{ E_{\mu }(
z^{-1};q^{-1},t^{-1}) \right\} $ can be computed explicitly in terms
of the quantities $d_{\eta }$ and $d'_{\eta }$. The derivation makes use of the the fact that for  $\eta _{i}<\eta _{i+1}$ we have \cite{sahi}
\begin{equation}
\frac{d_{s_{i}\eta }}{d_{\eta } }=\frac{%
1-\delta _{i,\eta } }{t-\delta _{i,\eta } 
}\;\;\;\text{and}\;\;\frac{d_{s_{i}\eta }^{\prime }}{d_{\eta
}^{\prime }}=\frac{t^{-1}-\delta _{i,\eta }}{%
1-\delta _{i,\eta }}.
\label{d si}
\end{equation}%
\begin{proposition}
\label{proposition 2.1}Let $\omega \in W_{I\cup J}$ be decomposed as in (\ref{12'}). Let $\omega \eta ^{\ast }=\mu $ and $\omega _{I}\eta ^{\ast }=\mu
_{I}.$ The coefficients in 
\begin{equation}
S_{\eta ^{\ast }}^{\left( I,J\right) }( z^{-1};q^{-1},t^{-1}) =\sum_{\mu \in
W_{I\cup J}\left( \eta ^{\ast }\right) }\widehat{b}_{\eta ^{\ast }\mu
}E_{\mu }( z^{-1};q^{-1},t^{-1}) ,\hspace{0.5cm }\widehat{b}_{\eta ^{\ast }\eta ^{\ast
}}=1,  \label{149}
\end{equation}%
are specified by%
\begin{equation}
\widehat{b}_{\eta ^{\ast }\mu }=\left( -1\right) ^{l\left( \omega
_{J}\right) }t^{l\left( \omega _{I}\right) }\frac{d_{\eta ^{\ast }}^{\prime
}d_{\mu }}{d_{\mu _{I}}^{\prime }d_{\mu _{I}}}. \label{21a}
\end{equation}
Similarly, the coefficients in 
\begin{equation}
S_{\eta ^{\ast }}^{\left( I,J\right) }\left( z\right) =\sum_{\mu \in
W_{I\cup J}\left( \eta ^{\ast }\right) }\widehat{c}_{\eta ^{\ast }\mu
}E_{\mu }\left( z\right) ,\hspace{0.5cm }\widehat{c}_{\eta ^{\ast }\eta ^{\ast }}=1,
\label{2.15}
\end{equation}%
are specified by 
\begin{equation}
\widehat{c}_{\eta ^{\ast }\mu }=\left( -\frac{1}{t}\right) ^{l\left( \omega
_{J}\right) }\frac{d_{\eta ^{\ast }}^{\prime }d_{\mu }}{d_{\mu _{I}}^{\prime
}d_{\mu _{I}}}.  \label{coefficients}
\end{equation}
\end{proposition}

\begin{proof}
We write%
\begin{eqnarray*}
\sum_{\mu \in W_{I\cup J}\left( \eta ^{\ast }\right) }\widehat{b}_{\eta
^{\ast }\mu }E_{\mu }( z^{-1};q^{-1},t^{-1}) \hspace{8cm} \\ =\sum_{\substack{ \mu \in W_{I\cup
J}\left( \eta ^{\ast }\right)  \\ \mu _{i}\leq \mu _{i+1}}}\chi
_{i,i+1}\left( \widehat{b}_{\eta ^{\ast }\mu }E_{\mu }( z^{-1};q^{-1},t^{-1}) +%
\widehat{b}_{\eta ^{\ast }s_{i}\mu }E_{s_{i}\mu }( z^{-1};q^{-1},t^{-1}) \right)
\end{eqnarray*}%
where $\chi _{i,i+1}=1/2$ if $\mu =s_{i}\mu $ and $1$ otherwise. For $i\in I$
we require 
\begin{equation}
T_{i}S_{\eta ^{\ast }}^{\left( I,J\right) }( z^{-1};q^{-1},t^{-1}) =tS_{\eta
^{\ast }}^{\left( I,J\right) }( z^{-1};q^{-1},t^{-1}) .  \label{symmetric prop}
\end{equation}%
If $\mu _{i}=\mu _{i+1}$ $\left( \ref{symmetric prop}\right) $ holds due to
the relation in $\left( \ref{TiEnInv1}\right) .$ Hence we consider the case
where $\mu _{i}<\mu _{i+1}.$ Expanding the left hand side of $\left( \ref%
{symmetric prop}\right) $ using $\left( \ref{TiEnInv1}\right) $ gives
simultaneous equations and solving these show%
\begin{equation}
\frac{\widehat{b}_{\eta ^{\ast }\mu }}{\widehat{b}_{\eta ^{\ast }s_{i}\mu }}=%
\frac{1-t\delta _{i,\mu }}{1-\delta _{i,\mu }}\hspace{0.5cm}\text{ for all }i\in I.
\label{first ratio}
\end{equation}%
Since $\mu _{i}<\mu _{i+1}$ $\left( \ref{d si}\right) $ can be used to
rewrite $\left( \ref{first ratio}\right) $ as%
\begin{equation}
\frac{\widehat{b}_{\eta ^{\ast }\mu }}{\widehat{b}_{\eta ^{\ast }s_{i}\mu }}%
=t\frac{d_{s_{i}\mu }^{\prime }}{d_{\mu }^{\prime }}. \label{29a}
\end{equation}%
By noting $\eta^*=\omega^{-1}_I\mu_I=s_{i_1}\ldots_{i_{l(\omega^{-1}_I)}}\mu_I$ where each $s_i$ interchanges increasing components we can apply (\ref{29a}) repeatedly to obtain
\begin{equation}
\frac{\widehat{b}_{\eta ^{\ast }\mu _{I}}}{\widehat{b}_{\eta ^{\ast }\eta
^{\ast }}}=\widehat{b}_{\eta ^{\ast }\mu _{I}}=t^{l\left( \omega _{I}\right) }\frac{d_{\eta ^{\ast }}^{\prime }}{%
d_{\mu _{I}}^{\prime }}, \label{29b}
\end{equation}%
where the first equality follows from the normalisation $\widehat{b}_{\eta ^{\ast }\eta ^{\ast }}=1$. 
\bigskip \\
To complete the derivation we require a formula for the ratio $\widehat{b}_{\eta ^{\ast }\mu }/\widehat{b}_{\eta ^{\ast }\mu _{I}}$. Since for all $j\in J$ we have
\begin{equation*}
T_{j}S_{\eta ^{\ast }}^{\left( I,J\right) }( z^{-1};q^{-1},t^{-1}) =-S_{\eta ^{\ast
}}^{\left( I,J\right) }( z^{-1};q^{-1},t^{-1}),
\end{equation*}%
applying the above methods gives $\widehat{b}_{\eta ^{\ast }\mu }/\widehat{b}_{\eta ^{\ast }s_{j}\mu }=-d_{\mu }/d_{s_{j}\mu },$
and consequently
\begin{equation}
\frac{\widehat{b}_{\eta ^{\ast }\mu }}{\widehat{b}_{\eta ^{\ast }\mu _{I}}}%
=\left( -1\right) ^{l\left( \omega _{J}\right) }\frac{d_{\mu }}{d_{\mu _{I}}}%
. \label{121}
\end{equation}%
Combining (\ref{121}) with $\left( \ref{29b}\right) $ we obtain $\left( \ref%
{21a}\right) .$
\bigskip \\
The derivation of (\ref{coefficients}) is as above, only replacing $\left( \ref{TiEnInv1}\right) $ with $%
\left( \ref{TiEn}\right)$.
\hfill $\square$ \end{proof}
We now use Proposition \ref{proposition 2.1} to determine $a_{\eta
}^{\left( I,J\right) }.$ To present the result requires some notation. Write 
$\eta ^{\left( \epsilon _{I},\epsilon _{J}\right) },$ where $\epsilon
_{I},\epsilon _{J}\in \left\{ +,0,-\right\} ,$ to denote the element of $%
W_{I\cup J}\left( \eta \right) $ with the properties that $\eta ^{\left(
+,\cdot \right) }$ $\left( \eta ^{\left( \cdot ,+\right) }\right) $ has $%
\eta _{i}^{\left( +,\cdot \right) }\geq \eta _{i+1}^{\left( +,\cdot \right)
} $ for all $i\in I$ ($\eta _{j}^{\left( +,\cdot \right) }>\eta
_{j+1}^{\left( +,\cdot \right) }$ for all $j\in J$), $\eta ^{\left( -,\cdot
\right) }$ $\left( \eta ^{\left( \cdot ,-\right) }\right) $ has $\eta
_{i}^{\left( -,\cdot \right) }\leq \eta _{i+1}^{\left( -,\cdot \right) }$
for all $i\in I$ ($\eta _{j}^{\left( -,\cdot \right) }<\eta _{j+1}^{\left(
-,\cdot \right) }$ for all $j\in J$), and $\eta ^{\left( 0,\cdot \right) }$
has $\eta _{i}^{\left( 0,\cdot \right) }=\eta _{i+1}^{\left( 0,\cdot \right)
}$ for all $i\in I.$ For example for $\mu=\omega_I \omega_J \eta^*$ we have $\omega_I \eta^*=\mu^{(0,+)}$ and $\omega_J \eta^*=\mu^{(+,0)}.$ Also introduce
\begin{equation*}
M_{I,\eta }:=\sum_{\substack{ \sigma ^{\prime }\in W_{I}  \\ \sigma ^{\prime
}\left( \eta \right) =\eta ^{\left( +,0\right) }}}t^{l\left( \sigma ^{\prime
}\right) }.
\end{equation*}

\begin{proposition}
\label{proposition 2.2} The proportionality constant $a_{\eta }^{\left( I,J\right) }$ in $%
\left( \ref{2.13}\right) $ is specified by
\begin{equation*}
a_{\eta }^{\left( I,J\right) }=\left( -1\right) ^{l\left( \omega _{J}\right)
}M_{I,\eta }\frac{d_{\eta }^{\prime }d_{\eta ^{\left( -,+\right) }}^{\prime
}d_{\eta ^{\left( -,+\right) }}}{d_{\eta ^{\left( 0,+\right) }}^{\prime
}d_{\eta ^{\left( 0,+\right) }}d_{\eta ^{\left( -,-\right) }}^{\prime }}
\label{2.21}
\end{equation*}%
where $\omega _{J}$ is such that $\omega _{J}\eta^{(+,+)} =\eta ^{\left( +,0\right)
}.$
\end{proposition}

\begin{proof}
Let $G\left( x,y\right) $ be defined by%
\begin{equation}
G\left( x,y\right) =\sum_{\eta \in W_{I\cup J}\left( \eta ^{\ast }\right) }%
\frac{d_{\eta}}{d_{\eta }^{\prime }}E_{\eta}\left( x;q,t\right) E_{\eta
}\left( y^{-1};q^{-1},t^{-1}\right).  \label{2.13b}
\end{equation}%
It follows from $\left( \ref{TiEn}\right) $ and $\left( \ref{TiEnInv1}%
\right) $ that $T_{i}^{\left( x\right) }G\left( x,y\right) =T_{i}^{\left( y\right) }G\left(
x,y\right) \text{ for }i\in I\cup J,$ and hence%
\begin{equation}
O_{I,J}^{\left( x\right) }G\left( x,y\right) =O_{I,J}^{\left( y\right)
}G\left( x,y\right) .  \label{2.23J}
\end{equation}%
By $\left( \ref{symmetric oij}\right) $ and $\left( \ref{asymmetric oij}%
\right) $ we have $O_{I,J}^{\left( y\right) }E_{\eta }\left(
y^{-1};q^{-1},t^{-1}\right) =b_{\eta }^{\left( I,J\right) }S_{\eta ^{\ast
}}^{\left( I,J\right) }\left( y^{-1};q^{-1},t^{-1}\right)$ for some $b_{\eta }^{\left(
I,J\right) }\in 
\mathbb{Q}
\left( q,t\right) .$ Hence substituting $\left( \ref{2.13b}\right) $ into $%
\left( \ref{2.23J}\right) $ and recalling $\left( \ref{2.13}\right) $ shows%
\begin{equation*}
S_{\eta ^{\ast }}^{\left( I,J\right) }\left( x\right) \sum_{\eta \in W_{I\cup
J}\left( \eta ^{\ast }\right) }\frac{d_{\eta }}{d^{\prime }_{\eta }}a_{\eta
}^{\left( I,J\right) }E_{\eta }\left( y^{-1}\right) =S_{\eta ^{\ast
}}^{\left( I,J\right) }\left( y^{-1}\right) \sum_{\eta \in W_{I\cup J}\left(
\eta ^{\ast }\right) }\frac{d_{\eta }}{d^{\prime }_{\eta }}b_{\eta }^{\left(
I,J\right) }E_{\eta }\left( x\right) .
\end{equation*}%
Using seperation of variables it follows that 
\begin{equation}
S_{\eta ^{\ast }}^{\left( I,J\right) }\left( y^{-1}\right) =a_{\eta ^{\ast
}}\sum_{\eta \in W_{I\cup J}\left( \eta ^{\ast }\right) }\frac{d_{\eta
}}{d^{\prime }_{\eta }}a_{\eta}^{\left( I,J\right) }E_{\eta }\left(
y^{-1}\right)  \label{2.24}
\end{equation}%
for some constant $a_{\eta ^{\ast }}.$ Equating coefficients for $\eta=\omega_I \omega_J \eta^*$ in $\left( \ref%
{2.24}\right) $ and $\left( \ref{149}\right) $ shows%
\begin{equation}
a_{\eta ^{\ast }}\frac{d_{\eta }}{d^{\prime }_{\eta }}a_{\eta }^{\left(
I,J\right) }=\left( -1\right) ^{l\left( \omega _{J}\right) }t^{l\left(
\omega _{I}\right) }\frac{d_{\eta ^{\ast }}^{\prime }d_{\eta }}{d_{\eta
^{(0,+)}}^{\prime }d_{\eta^{(0,+)}}}.  \label{2.25}
\end{equation}%
The identity $\left( \ref{2.25}%
\right) $ must hold for all $\eta \in W_{I\cup J\left( \eta ^{\ast }\right) }$%
, and in particular for $\eta=\eta^{{(-,-)}}$. For such a composition we can use (\ref{TiEn}) to show
\begin{equation*}
O_{I,J}E_{\eta ^{\left( -,-\right) }}\left( x\right) =\left( -1\right)
^{l\left( \omega _{J^{\prime }}\right) }M_{I,\eta }S_{\eta ^{\ast }}^{\left(
I,J\right) }\left( x\right),
\end{equation*}%
where $\omega _{J^{\prime}}^{-1}\eta ^{\left( -,-\right) }=\eta ^{\left(
-,+\right) }.$ Consequently
\begin{equation}
a_{\eta ^{\left( -,-\right) }}^{(I,J)}=\left( -1\right) ^{l\left( \omega
_{J^{\prime }}\right) }M_{I,\eta }.  \label{this}
\end{equation}%
Substituting $\left( \ref{this}\right) $ into $\left( \ref{2.25}\right) $
with $\eta =\eta ^{\left( -,-\right) }$ implies%
\begin{equation}
a_{\eta ^{\ast }}=\frac{t^{l\left( \omega _{I}\right) }}{M_{I,\eta }}\frac{%
d_{\eta ^{\ast }}^{\prime }d_{\eta ^{\left( -,-\right) }}^{\prime }}{d_{\eta
^{\left( -,+\right) }}^{\prime }d_{\eta ^{\left( -,+\right) }}}.
\label{2.26}
\end{equation}%
Substituting $\left( \ref{2.26}\right) $ in $\left( \ref{2.25}\right) $ gives the desired result.
\hfill $\square$ \end{proof}

\begin{corollary}
We have the evaluation formula%
\begin{equation}
S_{\eta ^{\ast }}^{(I,\emptyset )}( t^{\underline{\delta }}) =\frac{n_{I}}{a_{\eta ^{\ast }}^{\left( I,\emptyset \right) }}E_{\eta ^{\ast }}(
t^{\underline{\delta }}) =\frac{n_{I}}{M_{I,\eta ^{\ast }}}\frac{%
t^{l\left( \eta \right) }e_{\eta ^{\ast }}}{d_{\eta ^{\ast \left( -,0\right)
}}},  \label{only symmetrised}
\end{equation}%
where $n_{I}:=\Sigma _{\sigma \in W_{I}}t^{l\left( \sigma \right)
}=\Pi _{s}[ | \widetilde{I}_{s}|] _{t}!.$
\end{corollary}

\begin{proof}
Using $T_{i}f\left( t^{\underline{\delta }}\right) =tf\left( t^{\underline{%
\delta }}\right) $ the first equality of $\left( \ref{only symmetrised}%
\right) $ can be derived immediately from $\left( \ref{2.13}\right) .$
With $J=\emptyset $ and $\eta =\eta ^{\ast }$ Proposition \ref{proposition
2.2}, gives 
\begin{equation*}
a_{\eta ^{\ast }}^{\left( I,\emptyset \right) }=M_{I,\eta ^{\ast }}\frac{%
d_{\eta ^{\ast \left( -,0\right) }}}{d_{\eta ^{\ast }}}.
\end{equation*}%
Substituting this and the well known result (see e.g. \cite{macdonald})%
\begin{equation*}
E_{\eta }( t^{\underline{\delta }};q,t) =t^{l\left( \eta \right) }%
\frac{e_{\eta } }{d_{\eta } }
\end{equation*}
gives the final equality.
\hfill $\square$ \end{proof}
\bigskip 
We now move on to our first related result, deducing the form of Macdonald polynomials with prescribed symmetry in specific cases.

\section{Special Forms of the Prescribed Symmetry Polynomials}\label{special forms}
\subsection{The main result}

We begin by introducing some notation to simplify the labeling of the $S_{\eta^*}^{(I,J)}$ of interest. With $N_p:=\{n_1,n_2,\ldots,n_p\}$, let
\begin{eqnarray*}
(\kappa_{n_0},\delta_{N_p})&:=&(\kappa_1,\ldots, \kappa_{n_0},n_1-1,n_1-2,\ldots,1,0,\ldots,n_p-1,\ldots,1,0) \\
I^{n_0}&:=&\{1,\ldots, n_{0}-1\}
\end{eqnarray*}
and
\begin{equation}
J^{n_0,N_p}:=\cup
_{i=1}^{p}\left\{ \Sigma _{j=1}^{i}n_{j-1}+1,...,\Sigma_{j=1}^{i+1}n_{j-1}-1\right\} \label{JnoNp}
\end{equation}
For example with $\kappa_3=(3,3,2)$ and $N_p=\{4,2\}$, $(\kappa_3,\delta_{\{4,2\}})=\left(3,3,2,3,2,1,0,1,0)\right) $, $I^3=\{1,2\}$ and $J^{3,\{4,2\}}=\{4,5,6,8\}$.
\bigskip \\
Related to the set $J^{n_0,N_p}=J$ are the generalised Vandermonde products $\Delta ^{n_0,N_p}\left( z\right)$ and $\Delta _{t}^{n_0,N_p}\left( z\right)$, defined by 
\begin{equation*}
\Delta^{{n_0,N_p}}\left( z\right) :=\prod_{\beta =1}^{p}\prod
_{\substack{ \min \left( \widetilde{J}_{\beta }\right) \leq i \\<j \leq \max \left( 
\widetilde{J}_{\beta }\right) }}\left( z_{i}-z_{j}\right)
\label{crazy delta1}
\end{equation*}
and
\begin{equation*}
\Delta^{{n_0,N_p}}_t\left( z\right) :=\prod_{\beta =1}^{p}\prod
_{\substack{ \min \left( \widetilde{J}_{\beta }\right) \leq i \\<j \leq \max \left( 
\widetilde{J}_{\beta }\right) }}\left( z_{i}-t^{-1}z_{j}\right).
\label{crazy delta}
\end{equation*}
In the theory of Jack polynomials with prescribed symmetry, these being denoted by $S_{\eta^*}^{(I,J)}(z;\alpha)$, using the properties of the eigenoperators it was found that \cite{peter4} with $\eta^*=(\kappa_{n_0},\delta_{N_p})$, $(I,J)=(I_{n_0},J_{n_0,N_p})$ and $\kappa$ a partition such that $\kappa_1<\min(n_1,\ldots,n_p)$ 
\begin{equation}
S^{(I,J)}_{\eta^*}(z;\alpha)=\Delta^{n_0,N_p}\left( z\right)J_\kappa^{(p+\alpha)}(z_1,\ldots,z_{n_0}), \label{jackana}
\end{equation}
where $J_\kappa(z;\alpha)$ is a symmetric Jack polynomial. It was shown in \cite{exchangeterms}, using different eigenoperator properties to the Jack case, that the Macdonald analogue of (\ref{jackana}), with the ordering of the $t$-symmetric and the $t$-antisymmetric variables and partitions switched, holds when $p=1$. Our interest is in the Macdonald analogue of (\ref{jackana}) for $p\geq 1$. The stated related results, along with computational evidence, leads us to conjecture the following. 
\begin{conjecture}
\label{conjecture}Let $\eta^*=(\kappa_{n_0},\delta_{N_p})$, $(I,J)=(I^{n_0},J^{n_0,N_p})$ and $\kappa$ a partition such that $\kappa_1<\min(n_1,\ldots,n_p)$, then for $p\geq1$
\begin{equation}
S^{(I,J)}_{\eta^*}(z;q,t)=\Delta_t^{n_0,N_p}\left( z\right)P_\kappa(z_1,\ldots,z_{n_0};qt^p,t).\label{mac ana}
\end{equation}
\end{conjecture}
\bigskip 
We first prove a special case of the conjecture then consider the general case. 
\begin{theorem} \label{special eta thm}With $\eta^*=(0^{n_0},\delta_{N_p})$ and $(I,J)=(I^{n_0},J^{n_0,N_p})$ we have
\begin{equation*}
S^{(I,J)}_{\eta^*}(z;q,t)=\Delta_t^{n_0,N_p}\left( z\right).
\end{equation*}
\end{theorem}
\begin{proof}
The result follows from $S_{\eta^*}^{(I,J)}$ having leading term $z^{(0^{n_0},\delta_{N_p})}$ and the requirement that $S_{\eta^*}^{(I,J)}$ be $t$-antisymmetric with respect to $J^{n_0,N_p}$.
\hfill $\square$ \end{proof}
\bigskip
Due to the structure of the eigenoperator for the Macdonald polynomials the methods used in the Jack theory cannot be generalised to prove Conjecture \ref{conjecture}, similarly the proof in \cite{exchangeterms} only works for the one-block case with the antisymmetric variables before the symmetric. However, within \cite{peter4} a brief note is made on how one may show the following result
\begin{equation}
S_{\rho+\delta}(z;\alpha)=\Delta(z)J_\kappa^{(\alpha/(1+\alpha))}(z;\alpha), \label{jack marshall}
\end{equation}
where $S_{\rho+\delta}$ is the antisymmetric Jack polynomial, using the fact that 
\begin{equation*}
\widetilde{H}^{(C,Ex)}_\alpha \Delta f=\Delta \widetilde{H}^{(C,Ex)}_{(\alpha/(1+\alpha))}f,
\end{equation*}
We refer the reader to \cite{peter4} for the definition of the Jack polynomial eigenoperator $\widetilde{H}^{(C,Ex)}_\alpha$ and further details of the suggested method. Low order cases have indicated that this method can be generalised to prove (\ref{jackana}). Therefore, although the Macdonald analogue of (\ref{jack marshall}), was found by Marshall in \cite{marshallmacdonald} using the orthogonality properties of the Macdonald polynomials we give the alternative derivation as suggested by \cite{peter4} and give suggestions as to how it could be generalised to prove (\ref{mac ana}). Before stating the theorem we introduce the eigenoperator for the symmetric Macdonald polynomials \cite{peter3}
\begin{equation*}
D_{n}^{1}\left( q,t\right) :=t^{n-1}\sum_{i=1}^{n}Y_{i},
\end{equation*}%
explicitly, 
\begin{equation*}
D_{n}^{1}\left( q,t\right) P_\kappa(z)=c_{\eta
^{\ast }}P_\kappa(z),\hspace{0.5cm}c_{\eta ^{\ast }}\in \mathbb{Q}\left( q,t\right). 
\end{equation*}%
\begin{theorem}
\label{prescribed sym}We have 
\begin{equation}
S_{\kappa +\delta }(z;q,t)=\Delta _{t}\left( z\right) P_{\kappa }\left(
z;q,qt\right).  \label{sym presym relationship}
\end{equation}%
\end{theorem}
\begin{proof}
Since the unique symmetric eigenfunction of $D_{n}^{1}\left( q,qt\right)$ with leading term $m_{\kappa}$ (the monomial symmetric polynomial indexed by $\kappa$) is $P_{\kappa }\left(
z;q,qt\right)$, (\ref{sym presym relationship}) will hold if for any symmetric function $f(z)$
\begin{equation}
D_{n}^{1}\left( q,t\right) \Delta _{t}\left( z\right) f\left( z\right)
=\Delta _{t}\left( z\right) D_{n}^{1}\left( q,qt\right) f\left( z\right) .
\label{prescribed sym result}
\end{equation}
Hence, our task will be to prove (\ref{prescribed sym result}). We begin by deriving a more explicit form for the left hand side of (\ref{prescribed sym result}). Since $%
\Delta _{t}\left( z\right) f\left( z\right) $ is t-antisymmetric the left hand side can be rewritten as 
\begin{equation*}
\left( T_{1}...T_{n-1}\omega -tT_{2}...T_{n-1}\omega +...+\left( -t\right)
^{n-1}\omega \right) \Delta _{t}\left( z\right) f\left( z\right)
\end{equation*}%
which, by the definition of $\omega $ is equal to 
\begin{equation*}
\left( T_{1}...T_{n-1}-tT_{2}...T_{n-1}+...+\left( -t\right) ^{n-1}\right)
\Delta _{t}\left( qz_{n},z_{1},...,z_{n-1}\right) f\left(
qz_{n},z_{1},...,z_{n-1}\right) .
\end{equation*}%
For simplicity we let $\Theta _{m}=T_{m}...T_{n-1},$ $g_{k}=\Delta
_{t}\left( qz_{k},z_{1},...,z_{n}\right) f\left(
qz_{k},z_{1},...,z_{n}\right) $ and 
\begin{equation*}
T_{i} =\frac{\left( 1-t\right) z_{i+1}}{z_{i}-z_{i+1}}+\frac{tz_{i}-z_{i+1}%
}{z_{i}-z_{i+1}}s_{i}.
\end{equation*}%
We begin by deducing the coefficient, $\widehat{c}[k,m]$ say, of each $g_{k}$ after being
operated on by $\Theta _{m}$. For this to be non-zero we require $m\leq k$.  For $qz_{k}$ to
appear in the first position of $f$ we must take the term that has had $s_{i}$ act
on it for each $i=n-1,...,k+1$ and therefore%
\begin{equation*}
\widehat{c}[k,k]=\prod_{i=k+1}^{n}\frac{tz_{k}-z_{i}}{z_{k}-z_{i}}.
\end{equation*}%
It can be shown by (backward) induction on $m$ that for $m<k$%
\begin{equation*}
\widehat{c}[k,m]=\left( -1\right) ^{k-m}\frac{\left( t-1\right) z_{k}}{%
z_{m}-z_{k}}\prod_{i=m+1}^{k-1}\frac{tz_{i}-z_{k}}{z_{i}-z_{k}}%
\prod_{i=k+1}^{n}\frac{tz_{k}-z_{i}}{z_{k}-z_{i}}.
\end{equation*}%
We note that an important part of the inductive proof is to keep $\Delta $
and $f$ of the form $\Delta _{t}\left( qz_{k},z_{1},...,z_{n}\right) f\left(
qz_{k},z_{1},...,z_{n}\right).$ This is done by observing that $%
s_{i}f\left( z\right) =f\left( z\right) $ and 
\begin{equation*}
s_{i}\Delta _{t}\left( z\right) =\frac{tz_{i+1}-z_{i}}{tz_{i}-z_{i+1}}\Delta
_{t}\left( z\right) .
\end{equation*}%
To derive the coefficient of $\Delta _{t}\left(
qz_{k},z_{1},...,z_{n}\right) f\left( qz_{k},z_{1},...,z_{n}\right) $ in the
overall operator we must evaluate $\sum_{m=1}^{k}\left( -t\right) ^{m-1}\widehat{c}[k,m].$ This is done by proving
\begin{equation*}
\sum_{m=j}^{k}\left( -t\right) ^{m-1}\widehat{c}[k,m]=\left( -t\right)
^{j-1}\prod_{i=j}^{k-1}\frac{z_{k}-tz_{i}}{z_{i}-z_{k}}\prod_{i=k+1}^{n}%
\frac{tz_{k}-z_{i}}{z_{k}-z_{i}}
\end{equation*}%
inductively with a base case of $j=k-1.$ It follows that the coefficient of $%
f\left( qz_{k},z_{1},...,z_{n}\right) $ in $D_{n}^{1}\left( q,t\right)
\Delta _{t}\left( z\right) f\left( z\right) $ is 
\begin{equation}
\prod_{i=j}^{k-1}\frac{z_{k}-tz_{i}}{z_{i}-z_{k}}\prod_{i=k+1}^{n}\frac{%
tz_{k}-z_{i}}{z_{k}-z_{i}}\times \Delta _{t}\left(
qz_{k},z_{1},...,z_{n}\right).  \label{the coefficient}
\end{equation}%
By noting 
\begin{equation*}
\Delta _{t}\left( qz_{k},z_{1},...,z_{n}\right) =\prod_{i=j}^{k-1}\frac{%
qz_{k}-t^{-1}z_{i}}{z_{i}-t^{-1}z_{k}}\prod_{i=k+1}^{n}\frac{%
qz_{k}-t^{-1}z_{i}}{z_{k}-t^{-1}z_{i}}\Delta _{t}\left( z\right)
\end{equation*}%
we simpify $\left( \ref{the coefficient}\right) $ to%
\begin{equation*}
\Delta _{t}\left( z\right) \prod_{\substack{ i=1  \\ i\not=k}}^{n}\frac{%
qtz_{k}-z_{i}}{z_{k}-z_{i}},
\end{equation*}%
and hence%
\begin{equation*}
D_{n}^{1}\left( q,t\right) \Delta _{t}\left( z\right) f\left( z\right)
=\Delta _{t}\left( z\right) \sum_{k=1}^{n}\prod_{\substack{ i=1  \\ i\not=k}}%
^{n}\frac{qtz_{k}-z_{i}}{z_{k}-z_{i}}f\left( qz_{k},z_{1},...,z_{n}\right) .
\end{equation*}%
We now simplify the right hand side of (\ref{prescribed sym result}). We have 
\begin{eqnarray}
D_{n}^{1}\left( q,qt\right) f\left( z\right) &=&\left( qt\right)
^{n-1}\sum_{i=1}^{n}Y_{i}^{(q)}f\left( z\right) \notag\\
&=&\left( qt\right) ^{n-1}\left( \left( qt\right) ^{1-n}T_{1}^{\left(
q\right) }...T_{n-1}^{\left( q\right) }\omega +...+\omega T_{1}^{-1\left(
q\right) }...T_{n-1}^{-1\left( q\right) }\right) f\left( z\right)   \label{early step} 
\end{eqnarray}%
where $Y_{i}^{(q)},\text{ }T_{i}^{\left( q\right) },\text{ }T_{i}^{-1\left( q\right) }$ are
the operators $Y_{i},\text{ }T_{i},\text{ }T_{i}^{-1}$ with $t$ replaced by $qt.$ Since $f\left( z\right) $ is
symmetric we have $T_{i}^{-1\left( q\right) }f\left( z\right) =\left(
qt\right) ^{-1}f\left( z\right) .$ Using this and the action of $\omega $ $%
\left( \ref{early step}\right) $ simplifies to%
\begin{equation*}
(T_{1}^{\left( q\right) }...T_{n-1}^{\left( q\right) }+T_{n-1}^{\left(
q\right) }+1)f\left( qz_{n},z_{1},...,z_{n-1}\right) .
\end{equation*}%
We let $\Theta _{m}^{\left( q\right) }=T_{m}^{\left( q\right)
}...T_{n-1}^{\left( q\right) }$ and denote the coefficient of each $f\left(
qz_{k},z_{1},...,z_{n}\right) $ by $\widehat{c}_{q}[k,m]$ for each $m\leq k.$
Similarly to before 
\begin{equation*}
\widehat{c}_{q}[k,k]=\prod_{i=k+1}^{n}\frac{qtz_{k}-z_{i}}{z_{k}-z_{i}}
\end{equation*}%
and, by induction, 
\begin{equation*}
\widehat{c}_q[k,m]=\left( -1\right) ^{k-m}\frac{\left( qt-1\right) z_{k}}{%
z_{m}-z_{k}}\prod_{i=m+1}^{k-1}\frac{qtz_{k}-z_{i}}{z_{i}-z_{k}}%
\prod_{i=k+1}^{n}\frac{qtz_{k}-z_{i}}{z_{k}-z_{i}}.
\end{equation*}%
For use in (\ref{early step}) we require $\sum_{m=1}^{k}\widehat{c}_{q}[k,m].$ This is found by induction on 
\begin{equation*}
\sum_{m=j}^{k}\widehat{c}_q[k,m]=\prod_{\substack{ i=j  \\ i\not=k}}^{n}\frac{%
qtz_{k}-z_{i}}{z_{k}-z_{i}}.
\end{equation*}%
Therefore
\begin{equation*}
\Delta _{t}\left( z\right) D_{n}^{1}\left( q,qt\right) f\left(
z\right)=\Delta _{t}\left( z\right) \sum_{k=1}^{n}\prod_{\substack{ i=1  \\ i\not=k}}%
^{n}\frac{qtz_{k}-z_{i}}{z_{k}-z_{i}}f\left( qz_{k},z_{1},...,z_{n}\right),
\end{equation*}
which shows (\ref{prescribed sym result}), and consequently (\ref{sym presym relationship}), to be true.
\hfill $\square$ \end{proof}
\bigskip 
Trial cases suggest that for a symmetric function $f\left(
z_{1},..,z_{n_{0}}\right) $ with leading term $m_{\kappa }$ and $\kappa
_{1}<\min \left( n_{1},...,n_{p}\right) $ one has
\begin{equation}
D_{n}^{1}\left( q,t\right) \Delta _{t}^{n_0,N_p}\left( z\right)
f\left( z_{1},..,z_{n_{0}}\right) =\Delta _{t}^{n_0,N_p}\left(
z\right) D_{n}^{1}\left( qt^{p},t\right) f\left( z_{1},..,z_{n_{0}}\right). \label{dbig}
\end{equation}
We believe it to be possible to prove Conjecture \ref{conjecture} by first proving (\ref{dbig}). At this stage however it is not clear how one would keep track of the blocks of variables within the antisymmetrising set, making a strategy used to prove Theorem \ref{prescribed sym} problematic.
%

\subsection{A consequence of the conjecture}
A major result in the theory of Jack polynomials with prescribed symmetry
is the evaluation $U_{\eta ^{\ast }}^{\left( I,J\right) }\left(
1^{n};\alpha \right) $ \cite{dunkl}, where $U_{\eta ^{\ast }}^{\left( I,J\right) }\left(
z\right) $ is defined by%
\begin{equation*}
U_{\eta ^{\ast }}^{\left( I,J\right) }\left( z;\alpha\right) :=\frac{%
S_{\eta ^{\ast }}^{\left( I,J\right) }\left( z;\alpha\right) }{\Delta ^{n_0,N_p}\left( z\right) }.
\end{equation*}%
It does not appear possible to generalise the Jack method  to find the analogous Macdonald evaluation, $U_{\eta ^{\ast }}^{\left( I,J\right) }( t^{\underline{\delta }%
};q,t)$, where
\begin{equation*}
U_{\eta ^{\ast }}^{\left( I,J\right) }\left( z;q,t\right) :=\frac{%
S_{\eta ^{\ast }}^{\left( I,J\right) }\left( z;q,t\right) }{\Delta_t ^{n_0,N_p}\left( z\right) }.
\end{equation*}
by (\ref{mac ana}) use of the evaluation formula for the symmetric Macdonald polynomials \cite{macdonald} 
\begin{equation*}
P_{\kappa }( t^{\underline{\delta }};q,t) =\prod_{s\in\text{diag}(\kappa)}\frac{1-q^{a'_{\eta}(s)}t^{n-l'_{\eta}(s)}}{1-q^{a'_{\eta}(s)}t^{l'_{\eta}(s)+1}}%
\end{equation*} 
gives the following as a corollary to Conjecture \ref{conjecture}
\begin{equation*}
U_{\eta ^{\ast }}^{\left( I,J\right) }( t^{\underline{\delta }%
};q,t) =\prod_{s\in\text{diag}(\eta^*)}\frac{1-(qt^{p})^{a'_{\eta}(s)}t^{n-l'_{\eta}(s)}}{1-(qt^{p})^{a'_{\eta}(s)}t^{l'_{\eta}(s)+1}}.
\end{equation*} 
\section{The Inner Product of Prescribed Symmetry
Polynomials and Constant Term Identities\label{prescribed symmetry
polynomials and the inner product}}

\subsection{The inner product of prescribed symmetry polynomials}
We begin this section by finding the explicit formulas for the inner product of the prescribed symmetry polynomials
\begin{equation}
\left\langle S_{\eta ^{\ast }}^{\left( I,J\right) }\left( z\right) ,S_{\eta
^{\ast }}^{\left( I,J\right) }\left( z\right) \right\rangle _{q,t},
\label{presym inner}
\end{equation}
in terms of nonsymmetric Macdonald polynomials. We then proceed to show how these formulas can be used to prove specialisations of certain constant term conjectures.  We first consider the inner product of $O_{I,J}E_{\eta}$. 
\begin{lemma}
\label{lemma ratio inner}With $K_{I,J}(t):=\Pi _{i=1}^{s}[| \widetilde{J}_{j}|
]_{t}!\Pi _{i=1}^{r}[| \widetilde{I}_{i}| ]_{t^{-1}}!$ we have 
\begin{equation*}
\left\langle O_{I,J}E_{\eta }\left( z\right) ,O_{I,J}E_{\eta }\left(
z\right) \right\rangle _{q,t}=K_{I,J}(t)\left\langle O_{I,J}E_{\eta }\left(
z\right) ,E_{\eta }\left( z\right) \right\rangle _{q,t}
\label{oij proposition},
\end{equation*}%
where $I_{i}$ and $J_{j}$ denote the decomposition of $I$ and 
$J$ as a union of sets of consecutive integers.
\end{lemma}

\begin{proof}
We begin by rewriting the left hand side of $\left( \ref{oij proposition}%
\right) $ as 
\begin{eqnarray}
&&\left\langle O_{I,J}E_{\eta }\left( z\right) ,\sum_{\omega \in W_{I\cup
J}}\left( -\frac{1}{t}\right) ^{l\left( \omega _{J}\right) }T_{\omega }\left[
E_{\eta }\left( z\right) \right] \right\rangle _{q,t}  \notag \\
&=&\sum_{\omega \in W_{I\cup J}}\left\langle O_{I,J}E_{\eta }\left( z\right)
,\left( -\frac{1}{t}\right) ^{l\left( \omega _{J}\right) }T_{\omega }\left[
E_{\eta }\left( z\right) \right] \right\rangle _{q,t}.  \label{2nd line}
\end{eqnarray}%
Since $T_{i}^{-1}$ is the adjoint operator of $T_{i},$ that is $\left\langle f,T_{i}g\right\rangle _{q,t}=\left\langle T_{i}^{-1}f,g\right\rangle _{q,t},$ it follows that $\left( \ref{2nd line}\right) $ is equal to 
\begin{equation}
\sum_{\omega \in W_{I\cup J}}\left\langle T_{\omega }^{-1}O_{I,J}E_{\eta
}\left( z\right) ,\left( -\frac{1}{t}\right) ^{l\left( \omega _{J}\right)
}E_{\eta }\left( z\right) \right\rangle _{q,t}.  \label{nearly there}
\end{equation}%
Using $\left( \ref{symmetric oij}\right) $ and $\left( \ref{asymmetric oij}%
\right) $ we rewrite $\left( \ref{nearly there}\right) $ as%
\begin{equation*}
\sum_{\omega \in W_{I\cup J}}\left\langle \frac{\left( -1\right) ^{l\left(
\omega _{J}\right) }}{t^{l\left( \omega _{I}\right) }}O_{I,J}E_{\eta }\left(
z\right) ,\left( -\frac{1}{t}\right) ^{l\left( \omega _{J}\right) }E_{\eta
}\left( z\right) \right\rangle _{q,t}.
\end{equation*}%
By definition of the inner product $\left\langle \cdot ,t^{-1}\right\rangle
_{q,t}=\left\langle t, \cdot \right\rangle _{q,t}$ and therefore we have 
\begin{equation*}
\sum_{\omega \in W_{I\cup J}}\left\langle \frac{t^{l\left( \omega
_{J}\right) }}{t^{l\left( \omega _{I}\right) }}O_{I,J}E_{\eta }\left(
z\right) ,E_{\eta }\left( z\right) \right\rangle _{q,t} =\sum_{\omega \in W_{I\cup J}}\frac{t^{l\left( \omega _{J}\right) }}{%
t^{l\left( \omega _{I}\right) }}\left\langle O_{I,J}E_{\eta }\left( z\right)
,E_{\eta }\left( z\right) \right\rangle _{q,t}.
\end{equation*}%
The final result is obtained using the identity \cite{stanley} $\sum_{\sigma \in S_{n}}q^{l\left( \sigma \right) }=[n]_{q}!.$
\hfill $\square$ \end{proof}
We now present the main theorem of this subsection.
\begin{theorem} \label{innerstuff}
We have 
\begin{equation}
\left\langle S_{\eta ^{\ast }}^{\left( I,J\right) }\left( z\right) ,S_{\eta
^{\ast }}^{\left( I,J\right) }\left( z\right) \right\rangle
_{q,t}=K_{I,J}\left( t\right) \frac{\widehat{c}_{\eta ^{\ast }\eta }}{a_{\eta }\left(
q^{-1},t^{-1}\right) }\left\langle E_{\eta }\left( z\right) ,E_{\eta }\left(
z\right) \right\rangle _{q,t}.  \label{presym inner formula}
\end{equation}
\end{theorem}

\begin{proof}
Using $\left( \ref{2.13}\right) $ we are able to rewrite $\left( \ref{presym
inner}\right) $ as%
\begin{equation*}
\left\langle \frac{1}{a_{\eta }^{\left( I,J\right) }\left( q,t\right) }%
O_{I,J}E_{\eta }\left( z;q,t\right) ,\frac{1}{a_{\eta }^{\left( I,J\right)
}\left( q,t\right) }O_{I,J}E_{\eta }\left( z;q,t\right) \right\rangle _{q,t}
\end{equation*}%
by the definition of the inner product and Lemma \ref{lemma ratio inner} we
write this as 
\begin{equation}
\frac{K_{I,J}\left( t\right) }{a_{\eta }^{\left( I,J\right) }\left(
q,t\right) a_{\eta }^{\left( I,J\right) }\left( q^{-1},t^{-1}\right) }%
\left\langle O_{I,J}E_{\eta }\left( z;q,t\right) ,E_{\eta }\left(
z;q,t\right) \right\rangle _{q,t}.  \label{prop1}
\end{equation}%
Again using $\left( \ref{2.13}\right) $ we write $\left( \ref{prop1}\right) $%
\begin{equation*}
\frac{K_{I,J}\left( t\right) }{a_{\eta }^{\left( I,J\right) }\left(
q^{-1},t^{-1}\right) }\left\langle S_{\eta ^{\ast }}^{\left( I,J\right)
}\left( z\right) ,E_{\eta }\left( z;q,t\right) \right\rangle _{q,t}.
\end{equation*}%
By $\left( \ref{2.15}\right) $ and the orthogonality of the Macdonald
polynomials we get the desired result.
\hfill $\square$ \end{proof}

\subsection{Special cases of the prescribed symmetry inner product}
Following the theory of Jack polynomials \cite{exchangeterms} we were lead to finding explicit formulas for (\ref{presym inner}) in the special case where $\eta^*$ is of the form
\begin{equation}
\eta^{*}_{(n_{0};n_{1})}:=(0_{n_0},\delta_{n_1})=(0,\ldots,0,n_{1}-1,\ldots,1,0),\label{simple special eta}
\end{equation}
$I=\emptyset$ and $J=J_{n_0,n_1}$. Upon further inspection of this formula it was observed that the result could be used to provide an alternative derivation of a specialisation of a constant term conjecture from \cite{genweight}. 
\bigskip \\
Whilst working with the constant term identities it became apparent that a further conjecture in \cite{genweight} was related to the more general inner product formula where $\eta^{*}$ was given by 
\begin{equation}
\eta^{*}_{(n_{0};N_{p})}:={(0_{n_0};\delta_{N_{p}})}=(0,\ldots,0,n_{1}-1,\ldots,1,0,n_{2}-1,\ldots,1,0,\ldots,n_{p}-1,\ldots,1,0),\label{harder special eta}
\end{equation}
$I=\emptyset$ and $J=J_{n_0,N_p}$.
\bigskip \\
Using the theory developed in the previous section we give an explicit formula for \begin{equation*}
\left\langle S_{\eta ^{\ast }_{(n_{0};N_{p})}}^{\left( I,J\right) }\left( z\right) ,S_{\eta
^{\ast }_{(n_{0};N_{p})}}^{\left( I,J\right) }\left( z\right) \right\rangle _{q,q^{k}},
\end{equation*}
deriving the more specific formula for $\eta ^{\ast }= \eta ^{\ast }_{(n_{0};n_{1})}$ as a corollary.

\begin{theorem}
\label{lemma dp}With $\eta ^{\ast }=\eta _{\left( n_{0};N_{p}\right) }^{\ast
}$ as defined by (\ref{harder special eta}), $I=\emptyset$ and $J=J_{n_0,n_1}$ we have 
\begin{equation}
\left\langle S_{\eta ^{\ast }}^{(I,J)},S_{\eta ^{\ast }}^{(I,J)}\right\rangle _{q,q^{k}}
=\prod_{i=1}^{p}\left[ n_{i}\right] _{q^{k}}!q^{-k\Sigma _{i=1}^{p}\frac{%
n_{i}\left( n_{i}-1\right) }{2}}\frac{[kn+\max \left( N_{P}\right) ]_{q}!\left( 1-q\right) ^{\max \left( N_{P}\right) }}{%
[k]_{q}!^{n}\Pi
_{j=1}^{\max(N_{p})}\left( 1-q^{j}q^{k\left( n-m\left( j\right) \right)
}\right)}
\label{dp inner}
\end{equation}
where $m(j):=\sum_{k=i_j}^{p}(n_k^+-j),\text{ } i_j:=\#\{ n_k^+ \in N_p^+, n_k<j\}+1$, where $N_p^+:=\sigma \left(N_p\right)$ such that $n_1^+\geq \ldots \geq n_p^+$.
\end{theorem}

\begin{proof}
By (\ref{presym inner formula}) our task is to simplify
\begin{equation*}
K_{\emptyset,J}( q^{k}) \frac{\widehat{c}_{\eta ^{\ast }\eta }}{a_{\eta }\left(
q^{-1},q^{-k}\right) }\left\langle E_{\eta }\left( z\right) ,E_{\eta }\left(
z\right) \right\rangle _{q,q^{k}}.
\end{equation*}
For simplicity we take $\eta =\eta ^{\ast }$, and hence $\widehat{c}_{\eta ^{\ast }\eta ^{\ast }}=1.$ The $p$ disjoint sets in $J$ indicate
\begin{equation*}
K_{I,J}(q^k)=\prod_{i=1}^{p}\left[ n_{i}\right] _{q^{k}}!
\end{equation*}
and with $I=\emptyset$, $a_{\eta ^{\ast }}\left( q^{-1},q^{-k}\right)$ simplifies to 
\begin{equation*}
a_{\eta ^{\ast }}( q^{-1},q^{-k}) =\frac{d_{\eta ^{\ast
}}^{\prime }\left( q^{-1},q^{-k}\right) }{d_{\eta ^{\ast \left( -,-\right)
}}^{\prime }\left( q^{-1},q^{-k}\right) }.
\end{equation*}%
Lastly, by (\ref{12today}) and (\ref{13a}), we have 
\begin{equation*}
\left\langle E_{\eta ^{\ast }}( q,q^{k}) ,E_{\eta ^{\ast }}(
q,q^{k}) \right\rangle _{q,q^{k}}=\frac{d_{\eta^{*} }^{\prime }\left(
q,q^{k}\right) e_{\eta^{*} }\left( q,q^{k}\right) }{d_{\eta^{*} }\left( q,q^{k}\right) e_{\eta^{*}
}^{\prime }\left( q,q^{k}\right) }\frac{[nk]_{q}!}{[k]_{q}!^{n}}.
\end{equation*}%
Putting this all together allows us to rewrite (\ref{dp inner}) as
\begin{equation}
\prod_{i=1}^{p}\left[ n_{i}\right] _{q^{k}}!
\frac{d_{ \eta ^{\ast(-,-)}}^{\prime
}\left( q^{-1},q^{-k}\right) d_{\eta ^{\ast }}^{\prime }\left( q,q^{k}\right)
e_{\eta ^{\ast }}\left( q,q^{k}\right) }
{d_{\eta ^{\ast }}^{\prime }\left(
q^{-1},q^{-k}\right) d_{\eta ^{\ast }}\left( q,q^{k}\right)
e_{\eta ^{\ast }}^{\prime }\left( q,q^{k}\right) }\frac{[nk]_{q}!}{[k]_{q}!^{n}}
.  \label{big mess2}
\end{equation}%
We begin by simplifying 
\begin{equation}
\frac{d_{ \eta ^{\ast(-,-)}}^{\prime
}\left( q^{-1},q^{-k}\right) d_{\eta ^{\ast }}^{\prime }\left( q,q^{k}\right) }
{d_{\eta ^{\ast }}^{\prime }\left(
q^{-1},q^{-k}\right) d_{\eta ^{\ast }}\left( q,q^{k}\right) } \label{the ds}
\end{equation}
In comparison with $\eta^{*}$, $\eta ^{\ast \left( -,-\right)
}$ has one additional empty box above
each row. Hence the leg length of each $s\in $diag$(\eta ^{\ast \left( -,-\right)
}) $ is one greater than it's corresponding box in diag$%
\left( \eta ^{\ast }\right) .$ It follows from this and the definition of $%
d_{\eta }$ and $d_{\eta }^{\prime }$ that $d_{\eta^{\ast(-,-)}}=d_{\eta ^{\ast }}.$ Hence we can rewrite (\ref{the ds}) as
\begin{equation}
\frac{d_{\eta ^{\ast }}\left( q^{-1},q^{-k}\right) }{d_{\eta ^{\ast
}}^{\prime }\left( q^{-1},q^{-k}\right) }\frac{d_{\eta ^{\ast }}^{\prime
}\left( q,q^{k}\right) }{d_{\eta ^{\ast }}\left( q,q^{k}\right) }=q^{k\Sigma_{i=1}^{p}\frac{n_{i}(n_{i}-1)}{2}},\label{the ds simplified}
\end{equation}
where the equality follows upon use of the definition of $d^*$, given above (\ref{12today}), and the simple identity 
\begin{equation*}
\frac{1-x}{1-x^{-1}}=-x.  
\end{equation*}%
We now consider the simplification of the ratio $e$/$e^{\prime }$. Explicitly we have 
\begin{equation}
\frac{e_{\eta ^{\ast }}\left( q,q^{k}\right) }{e_{\eta ^{\ast }}^{\prime }\left(
q,q^{k}\right) }=\frac{\Pi _{s}\left( 1-q^{a_{\eta^* }^{\prime }\left( s\right)
+1}q^{k(n-l_{\eta^* }^{\prime }\left( s\right) )}\right) }{\Pi _{r}\left(
1-q^{a_{\eta^* }^{\prime }\left( r\right) +1}q^{k(n-1-l_{\eta^* }^{\prime }\left(
r\right)) }\right) }.  \label{try again}
\end{equation}%
In \cite{sahi} Sahi showed $e_\eta=e_{s_i\eta}$ and $e'_\eta=e'_{s_i\eta}$, hence, the products in (\ref{try again}) are independent of the row order. 
For simplicity we take $\eta^*=\eta^{*+}$. In such a composition $n-l'_{\eta^{*+}}((i,j))=n-1-l'_{\eta^{*+}}((i+1,j))$ and consequently most terms in (\ref{try again}) cancel. The terms unique to the numerator correspond to the boxes in the top row of $\eta^{*+}$, and therefore the terms remaining in the denominator will correspond to the bottom box of each column. The leg colengths of the latter set are given by $m(j)-1$, (for $j=1\ldots \max({N_p})-1$) . 
Hence the ratio of the $e^{\prime }$s in our
expansion is given by%
\begin{eqnarray*}
\frac{e_{\eta ^{\ast }}\left( q,q^{k}\right) }{e_{\eta ^{\ast }}^{\prime
}\left( q,q^{k}\right) } &=&\frac{\left( 1-q^{kn+1}\right) \left(
1-q^{kn+2}\right) ...\left( 1-q^{kn+ \max({N_p}) -1}\right) }{%
\Pi _{j=1}^{ \max({N_p})-1}\left( 1-q^{j}q^{k\left( n-m\left( j\right) \right)
}\right) } \notag \\
&=&\frac{\left( 1-q\right) ^{\max({N_p}) }[kn+ \max({N_p}) ]_{q}!}{\Pi
_{j=1}^{ \max({N_p})-1}\left( 1-q^{j}q^{k\left( n-m\left( j\right) \right)
}\right) \left( 1-q^{ \max({N_p}) }q^{kn}\right) \lbrack kn]_{q}!} \label{nearly} \\
&=&\frac{\left( 1-q\right) ^{ \max({N_p})}[kn+ \max({N_p}) ]_{q}!}{\Pi
_{j=1}^{ \max({N_p})}\left( 1-q^{j}q^{k\left( n-m\left( j\right) \right)
}\right) \lbrack kn]_{q}!}. \label{nearly2}
\end{eqnarray*}
The last simplification was made by noting $m(\max({N_p}))=0$. Substituting each simplification into $\left( \ref%
{big mess2}\right) $ gives the required result. 
\hfill $\square$ \end{proof}
We now give the analogous result for $\eta^{*}=\eta^{*}_{(n_{0};n_{1})}$.
\begin{corollary}
\label{interesting eta}With $\eta _{\left( n_{0};n_{1}\right) }^{\ast
}$ given by (\ref{simple special eta}) we have
\begin{equation}
\left\langle S_{\eta _{\left( n_{0};n_{1}\right) }^{\ast }}^{\left(
I,J\right) }\left( z\right) ,S_{\eta _{\left( n_{0};n_{1}\right) }^{\ast
}}^{\left( I,J\right) }\left( z\right) \right\rangle _{q,q^{k}}=\frac{%
[n_{1}]_{q^{k}}![n_{1}+nk]_{q}!\left( 1-q\right) ^{n_{1}}}{%
([k]_{q}!)^{n}\left( q^{n_{1}\left( k+1\right) +n_{0}k};q^{-\left(
k+1\right) }\right) _{n_{1}}q^{\frac{k\left( n_{1}-1\right) n_{1}}{2}}}.
\label{the answer}
\end{equation}
\end{corollary}
\begin{proof}
The result follows immediately from Theorem \ref{lemma dp} by substituting $p=1$ into the right hand side of (\ref{dp inner}) and simplifying using (\ref{9a}). \hfill $\square$ \end{proof}

\subsection{The constant term identities}
We now present the two conjectures put forward by Baker and Forrester \cite{genweight}. We conclude the paper by showing how our results can be used to prove the special case of these conjectures when $a=b=0$.

\begin{conjecture}\label{forresters conjecture} 
\cite[Conj 2.1]{genweight}  We have
\begin{eqnarray*}
D_{1}\left( n_{0};n_{1};a,b;q\right)&=&\text{\rm{CT}}\bigg(\prod_{n_{0}+1\leq i<j\leq n}(
z_{i}-q^{k+1}z_{j}) ( z_i^{-1}-q^{k}z_j^{-1})\prod_{i<j}\left( 
\frac{z_{i}}{z_{j}};q\right) _{k}\left( q\frac{z_{j}}{z_{i}};q\right) _{k}
  \notag \\
&& \times \prod_{i=1}^{n}\left(z_{i};q\right) _{a}\left(\frac{q}{z_{i}};q\right) _{b}\bigg)\notag \\
&=&\frac{\Gamma _{q^{k+1}}\left( n_{1}+1\right) }{\left( \Gamma _{q}\left(
1+k\right) \right) ^{n}}\prod_{l=0}^{n_{0}-1}\frac{\Gamma
_{q}\left( a+b+1+kl\right) \Gamma _{q}\left( 1+k(l+1\right) )}{\Gamma
_{q}\left( a+1+kl\right) \Gamma _{q}\left( b+1+kl\right) }  \notag \\
&&\times \prod_{j=0}^{n_{1}-1}\frac{\Gamma _{q}\left( \left( k+1\right)
j+a+b+kn_{0}+1\right) \Gamma _{q}\left( \left( k+1\right)
(j+1)+kn_{0}\right) }{\Gamma _{q}\left( \left( k+1\right)
j+a+kn_{0}+1\right) \Gamma _{q}\left( \left( k+1\right) j+b+kn_{0}+1\right) }.%
\label{line2}
\end{eqnarray*}%
\end{conjecture}
\bigskip %
The $n_0=0$ case of Conjecture \ref{forresters conjecture} reduces to the $q$-Morris constant term identity, well known in the theory of Selberg integrals (see, e.g., \cite{selberg}). Within \cite{genweight} Baker and Forrester were able to prove Conjecture %
\ref{forresters conjecture} for the cases $a=k$ and $n_{1}=2.$ In a related work \cite{baker} they also proved the case where $a=b=0$. In both cases a combinatorial identity of Bressaud and Goulden \cite{bressoud} is used. Following
this Hamada \cite{hamada} confirmed the general cases $n_{1}=2$ and $n_{1}=3
$ using a $q$-integration formula of Macdonald polynomials and Gessel \cite{gessel} showed the conjecture to be true for $n_{1}=2,n-1,3$ and also for
the cases where $n\leq 5.$ 
\bigskip 
\begin{conjecture}   \label{conj2}
\cite[Conj 2.2]{genweight} Define $D_{p}\left( n_{1},...,n_{p};n_{0};a,b,k;q\right)$ by
\begin{eqnarray*}
D_{p}\left( n_{1},...,n_{p};n_{0};a,b,k;q\right)&:=& \text{\rm{CT}}\bigg(\prod_{\alpha =1}^{p}\prod_{\substack{ \min \left( \widetilde{J}
_{\alpha }\right) \leq i< \\ j\leq \max \left( \widetilde{J}_{\alpha
}\right) }}(
z_{i}-q^{k+1}z_{j}) ( z_i^{-1}-q^{k}z_j^{-1}) \notag \\ 
&& \times \prod_{1\leq i<j\leq N}\left( \frac{z_{i}}{z_{j}};q\right)
_{k}\left( q\frac{z_{j}}{z_{i}};q\right) _{k}\prod_{i=1}^{n}\left( z_{i};q\right) _{a}\left( \frac{q}{z_{i}}%
;q\right) _{b} \bigg),  \label{Dp}
\end{eqnarray*}%
where $\widetilde{J}_{\alpha }$ is given by (\ref{JnoNp}). Then for $n_{p}>n_{j}$ $\left( j=1,...,p-1\right)$ we have
\begin{eqnarray}
\frac{D_{p}\left( n_{1},...,n_{p-1},n_{p}+1;n_{0};0,0,k;q\right) }{%
D_{p}\left( n_{1},...,n_{p-1},n_{p};n_{0};0,0,k;q\right) } &=&\frac{%
[n_{p}+1]_{q^{k+1}}}{[k]_{q}!} \notag \\
&&\times \frac{\Gamma_q( \left( k+1\right) \left( n_{p}+1\right) +k\Sigma
_{j=0}^{p-1}n_{j})}{\Gamma_q(\left( k+1\right) n_{p}+k\Sigma
_{j=0}^{p-1}n_{j})}  \notag \\
&=&\frac{[n_{p}+1]_{q^{k+1}}}{[k]_{q}!}\frac{[k\left( n+1\right)
+n_{p}]_{q}!}{[kn+n_{p}]_{q}!} \label{final corollary}
\end{eqnarray}%
\begin{eqnarray*}
\end{eqnarray*}
\end{conjecture}
Using the following Lemma reclaim the result for the $a=b=0$ case of Conjecture \ref%
{forresters conjecture} proved by Baker and Forrester in \cite{baker}. Although the result is already know to be true, the following highlights the strong connection between the conjectured constant terms and Macdonald polynomial theory. On this point the special case of Conjecture \ref{forresters conjecture} corresponding to the $q$-Morris identity is well known to relate to Macdonald polynomial theory and furthermore has generalisations involving the Macdonald polynomial in an identity due to Kaneko \cite{kaneko}.
\begin{lemma}
\label{kind of kadells}Let $J=\left\{ r,...,r+s\right\} \subseteq \left\{1,\ldots,n \right\}$ and $h\left(
z\right) $ be antisymmetric with respect to $z_{j},$ $j\in J$ then 
\begin{equation*}
\text{\rm{CT}}\bigg( \prod_{r\leq i<j\leq r+s}\left( z_{i}-az_{j}\right) h\left(
z\right) \bigg) =\frac{\left[ s\right] _{a}!}{s!}\text{\rm{CT}}\bigg(
\prod_{r\leq i<j\leq r+s}\left( z_{i}-z_{j}\right) h\left( z\right) \bigg) .
\label{kadells}
\end{equation*}
\end{lemma}

\begin{proof}
Let $S_{J}=\left\langle s_{j};j=r,...,r+s-1\right\rangle .$ For any
permutation $\sigma \in S_{J},$ the operation $z\rightarrow \sigma z$ leaves
the constant term unchanged. Hence%
\begin{equation*}
\text{CT}\bigg( \prod_{r\leq i<j\leq r+s}\left( z_{i}-az_{j}\right) h\left(
z\right) \bigg) =\text{CT}\bigg( \prod_{r\leq i<j\leq r+s}\left( z_{\sigma
\left( i\right) }-az_{\sigma \left( j\right) }\right) h\left( \sigma
z\right) \bigg) .
\end{equation*}%
Since $h\left( x\right) $ is antisymmetric $h\left( \sigma x\right) =\left(
-1\right) ^{l\left( \sigma \right) }h\left( x\right) ,$ summing over all
permutations gives%
\begin{equation*}
\text{CT}\bigg( \prod_{r\leq i<j\leq r+s}\left( z_{i}-az_{j}\right)
h\left( z\right) \bigg)=\frac{1}{s!}\text{CT}\bigg( \Big( \sum_{\sigma \in S_{J}}\left(
-1\right) ^{l\left( \sigma \right) }\prod_{r\leq i<j\leq r+s}\left(
z_{\sigma \left( i\right) }-az_{\sigma \left( j\right) }\right) \Big)
h\left( z\right) \bigg) .
\end{equation*}%
Since 
\begin{equation}
\sum_{\sigma \in S_{J}}\left( -1\right) ^{l\left( \sigma \right)
}\prod_{r\leq i<j\leq r+s}\left( z_{\sigma \left( i\right) }-az_{\sigma
\left( j\right) }\right) =\sum_{\omega \in S_{J}}\left( -1\right) ^{l\left( \omega \right) }\omega
\sum_{\sigma \in S_{J}}\left( -a\right) ^{l\left( \sigma \right) }z^{\sigma
\delta }. \label{finding constant term}
\end{equation}%
Letting $\omega=\sigma^{-1}$ we see that the constant term of (\ref{finding constant term}) is $\sum_{\sigma \in S_{J}}a^{l\left( \sigma \right)
}$. The result follows from the identity  \cite{stanley} $\sum_{\sigma \in S_{J}}a^{l\left( \sigma \right)
}=[s]_{a}!$.
\hfill $\square$ \end{proof}
\bigskip 
\begin{theorem} \label{big ct theorem}
\bigskip We have 
\begin{eqnarray}
D_{1}\left( n_{1};n_{0};0,0,k;q\right) &=&\frac{\Gamma _{q^{k+1}}\left(
n_{1}+1\right) }{\left( \Gamma _{q}\left( 1+k\right) \right) ^{n}}%
\prod_{j=0}^{n_{1}-1}\frac{\Gamma _{q}\left( \left( k+1\right)
(j+1)+kn_{0}\right) }{\Gamma _{q}\left( \left( k+1\right) j+kn_{0}+1\right) }
\notag \\
&&\times \prod_{l=0}^{n_{0}-1}\frac{\Gamma _{q}\left( 1+k(l+1\right) )}{%
\Gamma _{q}\left( 1+kl\right) }.  \label{line2a}
\end{eqnarray}
\end{theorem}

\begin{proof}
With $\eta ^{\ast }_{(n_{0},n_{1})}$ defined by (\ref{simple special eta}), Theorem \ref{special eta thm} gives
\begin{equation*}
S_{\eta ^{\ast }}\left( z\right) =\Delta_t^{n_0,\{n_1\}}(z)
\end{equation*}%
and hence 
\begin{eqnarray}
\left\langle S_{\eta ^{\ast }},S_{\eta ^{\ast }}\right\rangle _{q,q^{t}} =
\text{CT}\bigg(\prod_{ n_{0}+1\leq i<j\leq n}( z_{i}-q^{-k}z_{j}) ( z_i^{-1}-q^{k}z_j^{-1}) 
\prod_{1\leq i<j\leq n}\left( \frac{z_{i}}{z_{j}}%
;q\right) _{k}\left( q\frac{z_{j}}{z_{i}};q\right) _{k}\bigg).  \label{my d1}
\end{eqnarray}%
To be able to apply Lemma \ref{kind of kadells} we view $\left( \ref{my d1}\right) $ as
\begin{equation*}
\text{\rm{CT}}\bigg( \prod_{n_{0}+1\leq i<j\leq n}(
z_{i}-q^{-k}z_{j}) h\left( z\right) \bigg)
\end{equation*}
where $h\left( z\right) $ is antisymmetric with respect to $%
z_{n_{0}+1},...,z_{n}.$ Using Lemma \ref{kind of kadells} twice we
see that
\begin{equation*}
D_{1}\left( n_{1};n_{0};0,0,k;q\right) =\frac{\left[ n_{1}\right] _{q^{k+1}}!%
}{\left[ n_{1}\right] _{q^{-k}}!}\left\langle S_{\eta ^{\ast }},S_{\eta
^{\ast }}\right\rangle _{q,q^{t}}.
\end{equation*}%
Using (\ref{the answer}) and 
\begin{equation*}
\frac{\left[ n_{1}\right] _{q^{k}}!}{\left[ n_{1}\right] _{q^{-k}}!}=q^{
\frac{k\left( n_{1}-1\right) n_{1}}{2}}
\end{equation*} 
we obtain
\begin{equation}
D_{1}\left( n_{1};n_{0};0,0,k;q\right)=\frac{[n_{1}]_{q^{k+1}}![n_{1}+nk]_{q}!\left( 1-q\right)
^{n_{1}}}{([k]_{q}!)^{n}\left( q^{n_{1}\left( k+1\right)
+n_{0}k};q^{-\left( k+1\right) }\right) _{n_{1}}}. \label{D1 on the way}
\end{equation}
To show (\ref{D1 on the way}) is equivalent to (\ref{line2a}) we firstly use the identity \cite{qtheory} $\Gamma _{q}\left( 1+n\right) =[n]_{q}!$ to show 
\begin{equation*}
\frac{\lbrack n_{1}]_{q^{k+1}}!}{([k]_{q}!)^{n}}=\frac{\Gamma
_{q^{k+1}}\left( n_{1}+1\right) }{\left( \Gamma _{q}\left( 1+k\right)
\right) ^{n}}
\end{equation*}%
and%
\begin{equation*}
\prod_{l=0}^{n_{0}-1}\frac{\Gamma _{q}\left( 1+k(l+1\right) )}{\Gamma
_{q}\left( 1+kl\right) }=[n_{0}k]_{q}!.
\end{equation*}%
Lastly, to show
\begin{equation*}
\frac{\lbrack n_{1}+nk]_{q}!\left( 1-q\right) ^{n_{1}}}{\left(
q^{n_{1}\left( k+1\right) +n_{0}k};q^{-\left( k+1\right) }\right) _{n_{1}}}%
=[n_{0}k]_{q}!\prod_{j=0}^{n_{1}-1}\frac{[\left( k+1\right) j+kn_{0}+k]_{q}!%
}{\left[ \left( k+1\right) j+kn_{0}\right] _{q}!},
\end{equation*}%
we expand both sides and compare terms.
\hfill $\square$ \end{proof}

\bigskip 
Our final result in this section is proving the specialisation $a=b=0$ of Conjecture \ref{conj2}, which until now has not been done. We do this by finding the analogue of Theorem \ref{big ct theorem} for $\eta^{*}_{(n_{0},N_{p})}$ and stating the result as a corollary. We begin with a generalisation of Lemma \ref{kind of kadells}.
\begin{lemma}
\label{newer kadell}
Let $J\subseteq \left\{ 1,...n\right\} =J_{1}\cup \ldots \cup
J_{s}$ and $h\left( z\right) $ be antisymmetric with respect to $z_{j},$ $%
j\in J$ then 
\begin{equation*}
\text{\rm{CT}}\bigg(\prod_{\alpha =1}^{s}\prod_{\substack{ \min \left(
\widetilde{J}_{\alpha }\right) \leq i  \\ <j\leq \max \left( \widetilde{J}_{\alpha }\right) }}\left(
z_{i}-az_{j}\right) h\left( z\right) \bigg)\label{new kadells} 
=\prod_{\alpha =1}^{s}\frac{\left[ J_{\alpha }\right] _{a}!}{\left\vert
J_{\alpha }\right\vert !}\text{\rm{CT}} \bigg( \prod_{\alpha =1}^{s}\prod
_{\substack{ \min \left(\widetilde{J}_{\alpha }\right) \leq i  \\ <j\leq \max \left(\widetilde{J}_{\alpha }\right) }}\left( z_{i}-z_{j}\right) h\left( z\right) \bigg)  .
\end{equation*}
\end{lemma}
\bigskip 
\begin{theorem}
\label{corollary dp}We have%
\begin{eqnarray}
D_{p}\left( N_{p};n_{0};0,0,k;q\right) &=&\prod_{\alpha =1}^{p}\left( \left[
n_{\alpha }\right] _{q^{k+1}}!\right) \frac{[kn+\max \left( N_{P}\right)
]_{q}!}{[k]_{q}!^{n}}  \notag \\
&&\times \frac{\left( 1-q\right) ^{\max \left( N_{P}\right) }}{\Pi
_{i=1}^{\max(N_{p})}\left( 1-q^{i}q^{k\left( n-m\left( i\right) \right)
}\right)  }.
\label{dp formula}
\end{eqnarray}
\end{theorem}

\begin{proof}
With $\eta^{*}_{(n_{0},N_{p})}$ defined by (\ref{harder special eta}), Theorem \ref{special eta thm} implies
\begin{equation*}
S_{\eta ^{\ast }_{(n_{0},N_{p})}}\left( z\right) =\Delta_t^{n_0,N_p}\left( z\right)
\end{equation*}%
and hence the inner product $\left\langle S_{\eta ^{\ast }},S_{\eta ^{\ast
}}\right\rangle _{q,q^{k}}$ can be written as%
\begin{equation*}
\left\langle S_{\eta ^{\ast }},S_{\eta ^{\ast }}\right\rangle _{q,q^{k}}=%
\text{CT}\bigg( \prod_{\alpha =1}^{p}\prod_{\substack{ \min \left( 
\widetilde{J}_{\alpha }\right) \leq i<  \\ j\leq \max \left( \widetilde{J}%
_{\alpha }\right) }}( z_{i}-q^{-k}z_{j}) ( z_i^{-1}-q^{k}z_j^{-1})  \prod_{1\leq i<j\leq N}\left( \frac{z_{i}}{z_{j}}%
;q\right) _{k}\left( q\frac{z_{j}}{z_{i}};q\right) _{k}\bigg) .  \label{SDp}
\end{equation*}
By apply Lemma \ref{newer kadell} twice to $%
\left( \ref{dp inner}\right) $ and noting that 
\begin{equation*}
\frac{\left[ n_{i}\right] _{q^{k}}!}{\left[ n_{i}\right] _{q^{-k}}!}=q^{%
\frac{k\left( n_{i}-1\right) n_{i}}{2}}, 
\end{equation*}
the result is obtained using the methods of the derivation in Theorem \ref{big ct theorem}.
\hfill $\square$ \end{proof}
\bigskip
Conjecture \ref{conj2} is verified by substituting (\ref{dp formula}) into the left hand side of (\ref{final corollary}) and making the obvious simplifications.

\subsection*{Acknowledgement.} 
I would like to thank my supervisor Peter Forrester for his valuable advice and support and also Ole Warnaar for his comments on an earlier draft. This work was supported by an APA scholarship, the ARC and the University of Melbourne.
\bibliographystyle{plain}

\vspace{6cm}
\subsection*{}
\begin{flushright}
Wendy Baratta \\
Department of Mathematics \\
University of Melbourne \\
Australia \\
(E-mail: w.baratta@ms.unimelb.edu.au)
\end{flushright}

\end{document}